\documentclass{article}
\usepackage{graphicx} 

\title{WeylAlgeb}

\usepackage[english]{babel}

\usepackage[letterpaper,top=2cm,bottom=2cm,left=3cm,right=3cm,marginparwidth=1.75cm]{geometry}

\usepackage{amsmath, amssymb, latexsym}
\usepackage{graphicx}
\usepackage[colorlinks=true, allcolors=blue]{hyperref}
\usepackage{amsfonts}
\usepackage{amsthm}
\usepackage[usenames]{color}
\usepackage{indentfirst}
\usepackage{comment}

\usepackage{enumerate,xcolor}
\usepackage{amsrefs}

\newcommand{\di}{\displaystyle}

\newtheorem{theorem}{Theorem}

\newtheorem{corollary}[theorem]{Corollary}
\newtheorem{proposition}[theorem]{Proposition}
\newtheorem{lemma}[theorem]{Lemma}
\newtheorem{example}[theorem]{Example}
\newtheorem{remark}[theorem]{Remark}
\newtheorem{question}[theorem]{Question}

\newcommand{\add}{\operatorname{ad}}

\newcommand{\der}{{\rm Der}}
\newcommand{\Inn}{{\rm Inn}}
\newcommand{\aut}{{\rm Aut}}

\newcommand{\au}{{\rm Aut}}

\title{\textbf{\LARGE On Isotropy Groups of Quantum Weyl Algebra and Jordanian Plane}}
\author{ \small
ADRIANO DE SANTANA, RENE BALTAZAR, ROBSON VINCIGUERRA, WILIAN DE ARAUJO}
\date{\vspace{-5ex}}

\begin{document}
\Large{
\maketitle

\begin{abstract}
We study isotropy groups of $\sigma$-derivations of the quantum Weyl algebra and of ordinary derivations of the Jordanian plane. For the quantum Weyl algebra $A_q^1(\Bbbk)$, with $q$ not a root of unity, we use Brzeziński's classification to decompose every $\sigma$-derivation into inner and non-inner stable components. This yields an intersection formula for the isotropy group of an arbitrary $\sigma$-derivation and leads to explicit arithmetic descriptions. For the Jordanian plane $\Lambda_2(\Bbbk)$, we give a necessary and sufficient condition for an automorphism to belong to the isotropy group of an inner derivation. We compute the isotropy groups of monomial inner derivations and of locally nilpotent derivations. These examples show that isotropy groups in the Jordanian plane may contain
large triangular subgroups, unlike the quantum Weyl algebra. In this way, isotropy groups provide a natural invariant that reflects the structural difference between the Jordanian plane and the quantum Weyl algebra.
\end{abstract}

\section{Introduction}

Let $\Bbbk$ be an algebraically closed field of characteristic zero and let $A$ be a $\Bbbk$-algebra: our results remain valid over an arbitrary field of characteristic zero by
replacing each cyclic group $\mathbb Z_d$ with 
\[
\mu_d(\Bbbk)
=\{\lambda\in\Bbbk^\ast \, \vert \,\lambda^d=1\}.
\]

Given an automorphism $\sigma\in\aut(A)$, we denote by $\der_{\sigma}(A)$ the vector space of all $\sigma$-derivations of $A$, that is, the $\Bbbk$-linear maps $\delta:A\to A$ satisfying
\[
\delta(ab)=\delta(a)b+\sigma(a)\delta(b),
\quad a,b\in A.
\]

The centralizer of $\sigma$ in $\aut(A)$ is
\[
C_{\aut(A)}(\sigma)
=
\{\rho\in\aut(A)\, \vert \,\rho\sigma=\sigma\rho\}.
\]
If $\delta\in\der_{\sigma}(A)$, we define the isotropy group of $\delta$ by
\[
\aut_{\delta}^{\sigma}(A)
=
\{\rho\in C_{\aut(A)}(\sigma)\, \vert \, \rho\delta=\delta\rho\}.
\]
Equivalently,
\[
\aut_{\delta}^{\sigma}(A)
=
\{\rho\in C_{\aut(A)}(\sigma)\, \vert \, \rho\delta\rho^{-1}=\delta\}.
\]

The condition $\rho\sigma=\sigma\rho$ is necessary because, in general, if $\delta\in\der_{\sigma}(A)$ and $\rho\in\aut(A)$, then $\rho\delta\rho^{-1}$ is a $\rho\sigma\rho^{-1}$-derivation. Thus, conjugation preserves $\der_{\sigma}(A)$ precisely for automorphisms belonging to $C_{\aut(A)}(\sigma)$. When $\sigma= \text{id}$, we have $C_{\aut(A)}(\sigma)=\aut(A)$ and $\aut_{\delta}^{\sigma}(A)$ is the usual isotropy group of an ordinary derivation. Research on the isotropy subgroup of commutative $\Bbbk$-algebras has increased significantly in the past years. For example, I. Pan, L. Mendes, D. Levcovitz, L. Bertoncello, R. Baltazar, D. Yan, Y. Huang, M. Veloso, N. Dasgupta, A. Lahiri, A. Rittatore, S. Kour, H. Rewri  (see, \cites{Balta,BaltaVeloso,LevBert,Dasgupta,HuYan,PanBalta,PanMendes,Kour,Rittatore}.

The quantum Weyl algebra is the $\Bbbk$-algebra
\[
A_q^1(\Bbbk)
=
\Bbbk\langle x,y:yx=qxy+1\rangle,
\]
whereas the Jordanian plane is the $\Bbbk$-algebra
\[
\Lambda_2(\Bbbk)
=
\Bbbk\langle x,y:yx=xy+y^2\rangle.
\]

Quantum Weyl algebras and their generalizations have been studied from several points of view, including quantum groups, homological properties, polynomial identities, Hopf actions, automorphisms, and derivations. In addition,  for example, M. Alvarez, Q. Vivas, M. Almulhem, T. Brzezinski and L. Makar-Limanov (\cite{vivas,Limanov,bre}) have contributed significantly to understanding of these structures. The Jordanian plane has also been investigated in connection with automorphisms, derivations, module theory, prime ideals, noncommutative algebraic geometry, Hopf algebras,
and Nichols algebras; see, for example, \cite{Samuel,Shirikov}.

In this paper, we study isotropy groups of $\sigma$-derivations of the quantum Weyl algebra and of ordinary derivations of the Jordanian plane. For the quantum Weyl algebra, with parameter $q$ not a root of unity, we use Brzeziński's classification of $\sigma$-derivations to separate their inner and non-inner components. We determine the isotropy groups of the families occurring in this classification and obtain explicit arithmetic descriptions in terms of roots of unity. We then prove that every $\sigma$-derivation admits a decomposition into $\aut$-stable components and show that its isotropy group is the intersection of the isotropy groups of its nonzero components.

For the Jordanian plane, we first obtain a necessary and sufficient condition for an automorphism to belong to the isotropy group of an arbitrary inner derivation. We apply this criterion to inner derivations induced by monomials
and explicitly determine their isotropy groups. We also compute the isotropy groups of all locally nilpotent derivations, which are precisely the derivations of the form
\[
\delta_p(y)=0,
\qquad
\delta_p(x)=p(y),
\]
where $p(y)\in\Bbbk[y]$.

These computations reveal a sharp contrast between the two algebras. In the quantum Weyl algebra, the isotropy groups are constrained by its one-dimensional torus of automorphisms and are governed by arithmetic conditions. In the Jordanian plane, by contrast, isotropy groups may contain large triangular subgroups. This suggests that the collection of isotropy
groups of derivations can reflect structural properties of the underlying algebra and may be useful in distinguishing noncommutative algebras.

The previous classes of algebras come from the classic result: an Ore extension over a polynomial algebra $\Bbbk[x]$ is either a quantum plane, a quantum Weyl algebra, or an infinite-dimensional unital associative algebra $A_h$ generated by elements $x, y,$ that satisfy $yx-xy = h$, where $h \in \Bbbk[x]$ (for more details, see Lemma 1.2. \cite{Samuel}). This connection, together with the behavior of isotropy groups in the algebras considered in this paper, motivates the following question.

\begin{question}
Can one characterize the isotropy groups of derivations of Ore extensions of $\Bbbk[x]$?
\end{question}

The paper is organized as follows. In Section \ref{2}, we study
$\sigma$-derivations of the quantum Weyl algebra, determine the isotropy groups of the families appearing in Brzeziński's classification, and obtain the intersection formula for an arbitrary $\sigma$-derivation. In Section \ref{3}, we study derivations of the Jordanian plane, with particular emphasis on inner and locally nilpotent derivations, and compare the resulting isotropy groups with those occurring in the quantum and classical Weyl algebras.

\section{Quantum Weyl Algebra}\label{2}

Let $R=A_q^1(\Bbbk)$ be a quantum Weyl algebra where $q\in \Bbbk$, non-zero, is a not root of unity. Let $\sigma$ be an automorphism of $R$: by (\cite{vivas}, Theorem B), we have $\sigma(x)=\mu^{-1}x$ and $\sigma(y)=\mu y$, where $\mu \in \Bbbk^{\ast}$. 

Throughout this section, we fix the automorphism
\[
\sigma=\sigma_{\mu}\in\aut(R),
\quad
\sigma(x)=\mu^{-1}x,
\quad
\sigma(y)=\mu y,
\]
where $\mu\in\Bbbk^*$. Thus, all $\sigma$-derivations considered below are taken with respect to this fixed automorphism $\sigma$. On the other hand, whenever we write $\rho=\rho_{\beta}\in\aut(R)$, we mean an arbitrary automorphism of $R$. Since $q$ is not a root of unity, every automorphism of $R$ is of the form
\[
\rho=\rho_{\beta}\in\aut(R), \quad
\rho_{\beta}(x)=\beta^{-1}x,
\quad
\rho_{\beta}(y)=\beta y,
\]
for some $\beta\in\Bbbk^*$. Since both $\sigma$ and $\rho$ are diagonal automorphisms, they
commute: $\rho\sigma=\sigma\rho$. Consequently, $C_{\aut(R)}(\sigma)=\aut(R)$, and the isotropy group of a $\sigma$-derivation $\delta$ is
\[
\aut_{\delta}^{\sigma}(R)
=
\{\rho\in\aut(R) \, \vert \,\rho\delta=\delta\rho\}.
\]

	\begin{theorem}\label{brez}[\cite{bre},Theorem 6.2]
	 Assume that a non-zero $q \in \Bbbk$ is not a root of unity. Set $h=1-yx\in R$ and let $\mu$ be a non-zero element of $\Bbbk.$
	 \begin{enumerate}[a)]
	 \item\label{a}  For all $f(h) \in \Bbbk[h],$ the map $\delta$ on generators of $R$ given by
$$\delta(x)=f(h)x,\;\;\;\;\delta(y)=\mu^{-1}f(q^{-1}h)y$$
	 extends to a skew derivation $(\delta, \sigma_{\mu} )$ on R. These are the only $\sigma_{\mu}$-derivations such
that $\delta(h) = 0.$ They are inner if and only if there is no $d \in \{0,\ldots,deg(f)\}$ such that $\mu = q^{-d},$ and the coefficient $f_d$ in $\di{f (h) = \sum_ {k} f_{k} h^{k}}$ is not zero.
\item\label{b} If there exists $d \in \mathbb{N}$ such that
$\mu = q^{-d+1},$ then for all $a(x) \in \Bbbk[x]$ and $b(y) \in \Bbbk[y],$ the map given by
$$\delta(x) = h^{d} b(y),\;\;\;
 \delta(y) = h^{d} a(x),$$

extends to a skew derivation $(\delta, \sigma_{\mu})$ on $R$. All these derivations are inner if $d\neq 0,$ and they are not inner if $d = 0.$

\item The (combinations of the) above maps together with the inner-type derivations exhaust all $\sigma_{\mu}$-skew derivations on $R$ contained in (\cite{bre}, Theorem 3.1). Every $\sigma_{\mu}$-skew derivation on $R$ is of this type.

	 \end{enumerate}

	\end{theorem}

	\begin{proposition}\label{Pro0}
		If $\delta$ is a $\sigma$-derivation as in Theorem \ref{brez}.\ref{a}), then $\delta$ commutes with any automorphism of $R.$ 
	\end{proposition}
	\begin{proof}
	In fact,   let $\rho$ be a automorphism of $R$, so $\rho(x)=\beta^{-1} x$ and $\rho(y)=\beta y$. Then $\delta(\rho(x))=\delta(\beta^{-1}x)=\beta^{-1}\delta(x)=\beta^{-1}f(h)x$, on the other hand, as $\rho(h)=\rho(1-yx)=1-yx=h$ we have $\rho(\delta(x))=\rho(f(h)x)=\beta^{-1}f(h)x$, showing that $\delta(\sigma(x))=\sigma(\delta(x)).$
 
	Analogously, note that $\delta(\rho(y))=\beta\delta(y)=\beta\mu^{-1}f(q^{-1}h)y$ and
	$\rho(\delta(y))=\rho(\mu^{-1}f(q^{-1}h)y)=\mu^{-1}f(q^{-1}h)\rho(y)=\mu^{-1} \beta f(q^{-1}h)y$. Therefore, $\delta\circ\rho=\rho\circ\delta$.
	\end{proof}
\begin{lemma}\label{lema1}
    Let $\delta$ be a $\sigma$-derivation as in Theorem \ref{brez}.\ref{b}). Then $\delta$ commutes with an automorphism $\rho\neq id$ if, and only if, there is $n\in \mathbb{N}$ such that  $b(y)=\sum_{k} b_{kn+k-1}y^{kn+k-1}$ and $a(x)=\sum_{l} a_{ln+l-1}x^{ln+l-1}$.
\end{lemma}
\begin{proof}
    Let $b(y)=\sum b_{i}y^{i}\in \Bbbk[y]$ and $a(x)=\sum a_{j}x^{j}\in \Bbbk[x]$ be such that $\delta(x) = h^{d} b(y),\;\;\;
     \delta(y) = h^{d} a(x).$ Given $\rho$ be a automorphism of $R$, so $\rho(x)=\beta^{-1} x$ and $\rho(y)=\beta y$. We have $\rho(\delta(x))=\rho(h^{d} b(y))=h^{d} b(\rho(y))=h^{d} b(\beta y)$ and $\delta(\rho(x))=\delta(\beta^{-1} x)=\beta^{-1}\delta(x)=\beta^{-1}h^{d} b(y)$, since $\rho$ commute with $\delta$, we obtain $$\beta^{-1}=\beta^{i},$$ for all $i$ such that $b_{i}\neq 0$, i.e., $\beta^{i+1}=1$, for all $i$ such that $b_{i}\neq 0$, and therefore, other $\beta =1$ or there exist $n \in \mathbb{N}$ such that
     $b(y)=\sum_{k} b_{kn+k-1}y^{kn+k-1}$. Similarly, we also have that $\rho(\delta(y))=\rho(h^{d} a(x))=h^{d} a(\rho(x))=h^{d} a(\beta^{-1}x)$  and $\delta(\rho(y))=\delta(\beta y))=\beta \delta(y)=\beta h^{d} a(x)$. Since $\rho(\delta(y))=\delta(\rho(y))$, we must have that $(\beta^{j+1})^{-1}=1$, for all $j$ such that $a_{j}\neq 0$, and consequently, other $\beta =1$ or there exist $n \in \mathbb{N}$ such that
     $a(x)=\sum_{l} a_{ln+l-1}x^{ln+l-1}$.
     \end{proof}
     
\begin{proposition}\label{proposition2}
    Let $\delta$ be a $\sigma$-derivation as in the Lemma \ref{lema1}.  then the isotropy group of $\delta $ is isomorphic to $\mathbb{Z}_{d}$, where $d$ is the greatest commom divisor between all powers of $a(x) $ and $b(y)$  added to 1.
\end{proposition}
\begin{proof}
  By proof of Lemma \ref{lema1}, there is $n\in \mathbb{N}$ such that the powers of $a(x)$ and $b(y)$ are of the form $kn+k-1$. That is, if $j$ is an exponent of some monomial of $a(x)$ or $b(y)$, then $j+1 = k(n+1)$, and consequently,  $n+1$ is a common divisor. Let $d$ be the greatest common divisor between all powers of $a(x) $ and $b(y)$  added to 1, since $\beta^{j+1}=1$, we have $\beta^{d}=1$, and therefore the result follows.
\end{proof}

\begin{example} Consider the following the examples:

\begin{enumerate}
    \item Let $b(y)=y^{5} +y^{11}$ and $a(x)=x^{17}$, and so $d=6$. By Proposition \ref{proposition2}, we have that the isotropy group is isomorphic to $\mathbb{Z}_6$.
    \item Let $b(y)=y^{5} +y^{14}$ and $a(x)=x^{17}$, and so $d=3$. By Proposition \ref{proposition2}, we have that the isotropy group is isomorphic to $\mathbb{Z}_3$.
\end{enumerate}    
\end{example}

\begin{corollary}\label{cor0}
    Let $\delta$ be a $\sigma$-derivation as in Theorem \ref{brez}.\ref{b}). Let $d\in \mathbb{N}$ with $d\neq 1$, then $\sigma$ not in the isotropy group of $\delta.$
	\end{corollary}
 \begin{proof}
     Since $\delta $ is as Theorem \ref{brez}.\ref{b}) and $d\neq 1$ we have $\mu=q^{-d+1}$, that is, $\mu$ is not a root of unity, since $q$ is not a root of unity. If $\sigma$ belonged the isotropy groups of $\delta$, by Proposition \ref{proposition2}, $\mu$ is a root of unity, a contradiction.
 \end{proof}


Let $R$ be a quantum Weyl algebra and $\delta$ be a inner type $\sigma$-derivation.  Moreover, consider $\sigma: R \to R$ defined by $\sigma(x) = \mu^{-1}x$ and $\sigma(y) = \mu y$ and, for each $w \in R$, define $$I_w(a)=\add_w(a) = wa - \sigma(a)w,$$ for all $a \in R$. We also denote $w=\sum_{i,j} c_{(i,j)} x^iy^j\in R$ and  $\rho(x) = \beta^{-1}x$ and  $\rho(y) = \beta y$ by an automorphism of $R$.

\begin{theorem} \label{caso3}
   Let $w=\sum_{i}\sum_{j}c_{(i,j)} x^iy^j$ be in $R$. Let $S=\{j-i\in\mathbb Z: c_{(i,j)}\neq 0\}$. Let $\delta=\add_w$ be a $\sigma$-derivation. Then $\delta$ commutes with an automorphism $\rho\neq id$ if, and only if, $\beta^c=1$, where $c=\gcd S$.    
\end{theorem}

\begin{proof}
Define $s=j-i$, we can rewrite the polynomial $w$ this way
\begin{equation}
    \label{eq:w em duas}
    w=\sum_{s=-\infty}^\infty\sum_{i=0}^\infty c_{(i,s+i)}x^iy^{s+i}.
\end{equation}

The indices $s$'s for which there are some coefficients $c_{(i,s+i)}\neq 0$ are those such that $s\in S$. Therefore, we can rewrite this sum as
\begin{equation}
    \label{eq:w em duas e k finito}
    w=\sum_{s\in S}\sum_{i=0}^\infty c_{(i,s+i)}x^iy^{s+i}
\end{equation}
Note that, for each $s\in S$, there is a smaller index $i_s$ and a larger index $f_s$ such that $c_{(i,s+i)}\neq 0$ only for $i_s\le i\le f_s$. Therefore, we rewrite the sum as
\begin{equation}
    w=\sum_{s\in S}\sum_{i=i_s}^{f_s} c_{(i,s+i)}x^iy^{s+i}
\end{equation}
As $i_s$ is the smallest index with this property, we have $c_{(i_s,s+i_s)}\neq 0$. Likewise, we have that $c_{(f_s,s+f_s)}\neq 0$.

Let $S''$ be the subset of elements $s\in S$ such that $i_s=-s$ (this means that an element $s\in S''$ if and only if $x^{-s }$ is a monomial of $w$). Let $S'=S\setminus S''$. Thus, we have that $S=S'\cup S''$ is a disjoint union. Then we rewrite the polynomial $w$ by 
\begin{equation}
    \label{eq:w em duas finitos}
    w=\sum_{s\in S'}\sum_{i=i_s}^{f_s} c_{(i,s+i)}x^iy^{s+i}+\sum_{s\in S''}\sum_{i=i_s}^{f_s} c_{(i,s+i)}x^iy^{s+i}
\end{equation}

Note that, if $s\in S''$, then $i_s=-s$. Furthermore, if $s\in S'$, then $i_s>-s$. So, we rewrite Equation \eqref{eq:w em duas finitos} as
\begin{eqnarray}
    \label{eq:w em duas finitos 2}
    w&=&\sum_{s\in S'} c_{(i_s,s+i_s)}x^{i_s}y^{s+i_s}+
        \sum_{s\in S''} c_{(-s,0)}x^{-s}\\ \nonumber
     &&+\sum_{s\in S'}\sum_{i=i_s+1}^{f_s} c_{(i,s+i)}x^iy^{s+i}+
        \sum_{s\in S''}\sum_{i=i_s+1}^{f_s} c_{(i,s+i)}x^iy^{s+i}
\end{eqnarray}
The sum $\sum_{s\in S''} c_{(-s,0)}x^{-s}$ is the only one in the Equation \eqref{eq:w em duas finitos 2} where there are monomials whose exponent of $y$ is zero (i.e., it only depends on $x$).

Now, if $\delta\circ\rho-\rho\circ\delta=0$, then we have that
\begin{eqnarray}
    \label{eq:aplicado a x}
    0&=&\delta\circ\rho(x)-\rho\circ\delta(x)\\
     &=& \sum_{s\in S'} c_{(i_s,s+i_s)}\beta^{-1}(q^{s+i_s}-\mu^{-1})(\beta^s-1)x^{i_s+1}y^{s+i_s}+\nonumber\\
     & & \sum_{s\in S'} c_{(i_s,s+i_s)}\beta^{-1}(q^{i_s-1}+...+1)(\beta^s-1)x^{i_s}y^{s+i_s-1}+\nonumber\\
     & & \sum_{s\in S''}c_{(-s,0)}\beta^{-1}(q^0-\mu^{-1})(\beta^s-1)x^{-s+1}+\nonumber\\
     & & \sum_{s\in S'\cup S''}\sum_{i=i_s+1}^{f_s} c_{(i,s+i)}\beta^{-1}(q^{s+i}-\mu^{-1})(\beta^s-1)x^{i+1}y^{s+i}+\nonumber\\
     & & \sum_{s\in S'\cup S''}\sum_{i=i_s+1}^{f_s}\beta^{-1}(q^{i-1}+...+1)(\beta^s-1)x^{i}y^{s+i-1}\nonumber
\end{eqnarray}

Let $s\in S'$. Note that the only monomial in $x^{i_s}y^{s+i_s-1}$ appears only once in the second sum in the Equation \eqref{eq:aplicado a x} with the coefficient \linebreak $c_{(i_s,s+i_s )}\beta^{-1}(q^{i_s-1}+...+1)(\beta^s-1)$.
In fact, the same monomial cannot appear again in the second sum in the Equation \eqref{eq:aplicado a x}. If $x^{i_s}y^{s+i_s-1}$ appears in the first sum, then there exists $t \in S'$ such that $x^{i_s}y^{s+i_s-1} = x^{i_{t}+1}y^{t+i_{t}}$, but this implies that $s = t$ and $i_s = i_{t}+1$, contradicting the minimality of $i_s$. If $x^{i_s}y^{s+i_s-1}$ appears in any of the other sums, then there exists $t \in S' \cup S''$ and $i_{t} < i \le f_{t}$ such that $t$ and $i$ appear in the exponents of $x$ and $y$ in these sums. In all these cases, we have $s = t$. We have already seen that $t \in S'$ leads to a contradiction due to the minimality of $i_s$. On the other hand, $t \in S''$ leads to a contradiction with $S' \cap S'' = \emptyset$.

By the definition of $i_s$, we have that $c_{(i_s,s+i_s)}\neq 0$. Furthermore $(q^{i_s-1}+...+1)\neq 0$, otherwise $q$ would be a root of unity. Consequently, we have that $\beta^s-1=0$.

On the other hand
\begin{eqnarray}
    \label{eq:aplicado a y}
    0&=&\delta\circ\rho(y)-\rho\circ\delta(y)\\
     &=&\sum_{s\in S'} c_{(i_s,s+i_s)}\beta(1-\mu q^{i_s})(\beta^s-1)x^{i_s}y^{s+i_s+1}+ \nonumber\\
     & &\sum_{s\in S'} c_{(i_s,s+i_s)}\beta(q^{i_s+1}+...+1)(\beta^s-1)x^{i_s-1}y^{s+i_s}+\nonumber\\
     & &\sum_{s\in S''}c_{(-s,0)}\beta(1-\mu q^{-s})(\beta^s-1)x^{-s}y+\nonumber\\
     & &\sum_{s\in S''}c_{(-s,0)}\beta(q^{-s+1}+...+1)(\beta^s-1)x^{-s-1}+\nonumber\\
     & &\sum_{s\in S'\cup S''}\sum_{i=i_s+1}^{f_s} c_{(i,s+i)}\beta(1-\mu q^i)(\beta^s-1)x^{i}y^{s+i+1}+\nonumber\\
     & &\sum_{s\in S'\cup S''}\sum_{i=i_s+1}^{f_s} c_{(i,s+i)}\beta(q^{i+1}+...+1)(\beta^s-1)x^{i-1}y^{s+i}\nonumber
\end{eqnarray}

Let $s\in S''$. Let us note that the only monomial in $x^{-s-1}$ appears only once in the fourth summation with the coefficient \linebreak $c_{(-s,0)}\beta(q^ {-s+1}+...+1)(\beta^s-1)$. By the definition of $i_s$ (which in this case $i_s=-s$), we have that $c_{(-s,0)}\neq 0$. Furthermore $(q^{i_s-1}+...+1)\neq 0$, otherwise $q$ would be a root of unity. Consequently, we have that $\beta^s-1=0$.

We thus conclude that $\beta^s=1$ for all $s\in S'\cup S''=S$, which implies that $\beta^{\,\gcd S}=1$.
\end{proof}

The following result presents an arithmetic characterization for the isotropy group.

\begin{corollary}\label{resumo}
The isotopry group of $\delta=\add_w$ is isomorphic:
\begin{enumerate}[a)]
    \item $\Bbbk^{\ast}$, if we consider $\gcd(0,0)=0$ and $w=\sum_{i}c_{(i,i)} x^iy^i$;
    \item $\mathbb{Z}{_d}$, if $d=\gcd S$ and $d \in \mathbb{N} \setminus \{0,1\}$;
    \item $\{id\}$, if $1=\gcd S$.
\end{enumerate}
\end{corollary}

\begin{example} Using the notations of Theorem \ref{caso3}, we can easily construct examples that verify the Corollary \ref{resumo}, as follows:
\begin{enumerate}
    \item Let $w=\sum_{i \in L}c_{(i,i)} x^iy^i$ where $L$ is a set of indices. By definition of $S$, we have $S = \{0\}$ and, consequently, $\gcd S = 0$. According to Corollary \ref{resumo}.a), we can conclude that the isotopry group of $\delta=\add_w$ is isomorphic to $\Bbbk^{\ast}$.
    \item We can construct an example reaching Corollary \ref{resumo}.b) by taking $w = xy^4 + x^7 y^4 + x^9 y^6 + x^9 y^{12} + x^5y^{17}$, then $S = \{-3,3,12\}$, and so, $\gcd S = 3$. Applying the Corollary \ref{resumo}.b), the isotopry group of $\delta=\add_w$ is isomorphic to $\mathbb{Z}{_3}$.
    \item Finally, for $w = x^2y^2 + xy^4 + xy^5$, we get $S=\{0,1,4\}$. In this case, $\gcd S = 1$ and, by Corollary \ref{resumo}.b), the isotopry group of $\delta=\add_w$ is isomorphic to $\{id\}$. 
\end{enumerate}
    
\end{example}

We have, so far, determined the isotropy groups of the individual families of $\sigma$-derivations appearing in Theorem \ref{brez}. We combine these results in order to describe the isotropy group of an arbitrary $\sigma$-derivation of the quantum Weyl algebra. Before returning to the quantum Weyl algebra, we state the following general lemma, for an arbitrary $\Bbbk$-algebra $A$. If
\[
G\subseteq C_{\aut(A)}(\sigma)
\]
and $\delta\in\der_\sigma(A)$, we denote by
\[
\aut_\delta^G(A)
=
\{\rho\in G \, \vert \, \rho\delta\rho^{-1}=\delta\}
=
G\cap\aut_\delta^\sigma(A),
\]
the isotropy group of $\delta$ with respect to the conjugation action of $G$.

\begin{lemma}\label{directsum}
Let $G\subseteq C_{\aut(A)}(\sigma)$ be a subgroup acting on $\der_{\sigma}(A)$ by conjugation. Suppose that
$\mathcal D=\mathcal D_1\oplus\cdots\oplus \mathcal D_r
$ is a direct sum of $G$-stable subspaces of $\der_{\sigma}(A)$. Let $\delta=\delta_1+\cdots+\delta_r$, with $\delta_i\in\mathcal D_i$. Then
\[
\aut_{\delta}^{G}(A)
=
\bigcap_{\delta_i\neq0}\aut_{\delta_i}^{G}(A).
\]

\end{lemma}

\begin{proof}
The inclusion
\[
\bigcap_{\delta_i\neq0}\aut_{\delta_i}^{G}(A)
\subseteq
\aut_{\delta}^{G}(A)
\]
is immediate. Conversely, let $\rho\in\aut_{\delta}^{G}(A)$. Since each $\mathcal D_i$ is $G$-stable, we have $\rho\delta_i\rho^{-1}\in\mathcal D_i$
for every $i$. Moreover,
\[
\delta
=
\rho\delta\rho^{-1}
=
\sum_{i=1}^{r}\rho\delta_i\rho^{-1}.
\]

Then, $\sum_{i=1}^{r}
\left(\rho\delta_i\rho^{-1}-\delta_i\right)=0$. The summand $\rho\delta_i\rho^{-1}-\delta_i$ belongs to $\mathcal D_i$. Since the sum is direct, each summand must be zero. Therefore, $\rho\delta_i\rho^{-1}=\delta_i$ for every $i$ such that $\delta_i\neq0$. Thus,
\[
\rho\in
\bigcap_{\delta_i\neq0}\aut_{\delta_i}^{G}(A).
\]
\end{proof}

We recall that for each $w\in R$, define $I_w:R\to R$ by
\[
I_w(a)=wa-\sigma(a)w,
\quad a\in R.
\]

Then, $I_w$ is a $\sigma$-derivation, called the inner $\sigma$-derivation
induced by $w$. We also denote
\[
\Inn_\sigma(R)=\{I_w:w\in R\}
\subseteq\der_\sigma(R)
\]
the subspace of inner $\sigma$-derivations. 

We separate the families appearing in Theorem \ref{brez} into their inner and non-inner components. We begin with the family in Theorem \ref{brez}\ref{a}. Let
\[
f(h)=\sum_{k\geq0}f_kh^k
\]
and denote by $\delta_f$ the corresponding $\sigma$-derivation. By Theorem \ref{brez}\ref{a}, $\delta_f$ fails to be inner precisely when there exists $r\geq0$ such that
$\mu=q^{-r}$ and $f_r\neq0$. For $r\geq0$, let $H_r$ denote the $\sigma$-derivation associated to $f(h)=h^r$,
\[
H_r(x)=h^rx,
\qquad
H_r(y)=\mu^{-1}(q^{-1}h)^ry.
\]
If $\mu=q^{-r}$, then every derivation in the family
\ref{brez}\ref{a} can be written as the sum of an inner
$\sigma$-derivation and a scalar multiple of $H_r$. We, thus, define
\[
\mathcal H_\mu=
\begin{cases}
\Bbbk H_r, & \text{if } \mu=q^{-r}\text{ for some }r\geq0,\\
0, & \text{otherwise.}
\end{cases}
\]

We consider the family in Theorem \ref{brez}\ref{b}. Such a family occurs when $\mu=q^{-d+1}$, for some $d\geq0$. If $d\neq0$, Theorem \ref{brez}\ref{b} shows that every
derivation in this family is inner. Then, this contribution is already contained in $\Inn_\sigma(R)$. The only case in which this family contributes non-inner derivations is $d=0$, equivalently, $\mu=q$. In this case the derivations have the form $\delta(x)=b(y)$, $\delta(y)=a(x)$, where $a(x)\in\Bbbk[x]$ and $b(y)\in\Bbbk[y]$. Thus, we define
\[
\mathcal P_\mu
=
(\oplus_{i\geq0}\Bbbk X_i)
\oplus
(\oplus_{j\geq0}\Bbbk Y_j),
\]
where
\[
X_i(x)=y^i,
\qquad
X_i(y)=0,
\]
and
\[
Y_j(x)=0,
\qquad
Y_j(y)=x^j.
\]

Thus, when $\mu=q$, every derivation in this family is a finite linear combination of the $X_i$ and $Y_j$. If $\mu\neq q$, we denote $\mathcal P_\mu=0$.

\begin{proposition}\label{reduceddecomposition}
Every $\sigma$-derivation of $R$ can be written uniquely in the form
\[
\delta=I_w+\eta+\xi,
\]
where $I_w\in\Inn_\sigma(R)$, $\eta\in\mathcal H_\mu$ and $\xi\in\mathcal P_\mu$. Equivalently,
\[
\der_\sigma(R)
=
\Inn_\sigma(R)
\oplus
\mathcal H_\mu
\oplus
\mathcal P_\mu.
\]

Moreover, each summand is stable under the conjugation action of $\aut(R)$.
\end{proposition}

\begin{proof}
By Theorem \ref{brez}, every $\sigma$-derivation is the sum of an inner $\sigma$-derivation and contributions from the families described in Theorem \ref{brez}\ref{a} and \ref{brez}\ref{b}.

In the family \ref{brez}\ref{a}, all terms are inner except, possibly, the component corresponding to $h^r$ when $\mu=q^{-r}$. This remaining component belongs to $\mathcal H_\mu$. In the family \ref{brez}\ref{b}, all derivations are inner when $d\neq0$, whereas for $d=0$, equivalently $\mu=q$, they form
precisely the space $\mathcal P_\mu$. Then,
\[
\der_\sigma(R)
=
\Inn_\sigma(R)+\mathcal H_\mu+\mathcal P_\mu.
\]

The sum is direct because every nonzero element of $\mathcal H_\mu$ or
$\mathcal P_\mu$ is non-inner. Moreover, $\mathcal H_\mu$ and
$\mathcal P_\mu$ cannot both be nonzero: otherwise,
\[
\mu=q^{-r}=q
\]
for some $r\geq0$, which imply $q^{r+1}=1$, contrary to the assumption that $q$ is not a root of unity. Thus, the
decomposition is unique.

Finally, if $\rho=\rho_\beta\in\aut(R)$, then
\[
\rho I_w\rho^{-1}=I_{\rho(w)},
\]
and so $\Inn_\sigma(R)$ is stable. By Proposition \ref{Pro0},
\[
\rho H_r\rho^{-1}=H_r,
\]
and therefore $\mathcal H_\mu$ is stable. Moreover,
\[
\rho X_i\rho^{-1}
=
\beta^{i+1}X_i,
\quad
\rho Y_j\rho^{-1}
=
\beta^{-(j+1)}Y_j,
\]
which proves that $\mathcal P_\mu$ is also stable.
\end{proof}

Combining Lemma \ref{directsum} with Proposition
\ref{reduceddecomposition}, we can determine the isotropy group of an arbitrary $\sigma$-derivation by intersecting the isotropy groups of its nonzero components, as described in Propositions \ref{Pro0} and \ref{proposition2}, Theorem \ref{caso3}, and Corollary \ref{resumo}.

\begin{corollary}\label{intersection}
Let $\delta=I_w+\eta+\xi\in\der_\sigma(R)$ be the decomposition given by Proposition \ref{reduceddecomposition}, where $I_w\in\Inn_\sigma(R)$, $\eta\in\mathcal H_\mu$ and $\xi\in\mathcal P_\mu$. Then,
\[
\aut_\delta^\sigma(R)
=
\aut_{I_w}^\sigma(R)
\cap
\aut_\eta^\sigma(R)
\cap
\aut_\xi^\sigma(R),
\]
where the zero components are omitted from the intersection. Moreover:
\begin{enumerate}
\item $\aut_\eta^\sigma(R)=\aut(R)$;

\item if $\xi\neq0$, then $\xi$ is of the form
\[
\xi(x)=b(y),
\qquad
\xi(y)=a(x),
\]
and its isotropy group is the one described in Proposition
\ref{proposition2};

\item the isotropy group of the inner component $I_w$ is the one described in
Theorem \ref{caso3} and Corollary \ref{resumo}.
\end{enumerate}

In particular,
\[
\aut_\delta^\sigma(R)
=
\aut_{I_w}^\sigma(R)
\cap
\aut_\xi^\sigma(R).
\]

\end{corollary}

\begin{proof}
By Proposition \ref{reduceddecomposition}, we have the direct sum decomposition
\[
\der_\sigma(R)
=
\Inn_\sigma(R)
\oplus
\mathcal H_\mu
\oplus
\mathcal P_\mu,
\]
whose summands are stable under conjugation by $\aut(R)$. Since
$C_{\aut(R)}(\sigma)=\aut(R)$, Lemma \ref{directsum}, applied to $\delta=I_w+\eta+\xi$, gives
\[
\aut_\delta^\sigma(R)
=
\aut_{I_w}^\sigma(R)
\cap
\aut_\eta^\sigma(R)
\cap
\aut_\xi^\sigma(R),
\]
with the zero components omitted. The descriptions of the three factors follow respectively from Proposition
\ref{Pro0}, Proposition \ref{proposition2}, Theorem \ref{caso3}, and Corollary \ref{resumo}. In particular, Proposition \ref{Pro0} gives $\aut_\eta^\sigma(R)=\aut(R)$.

Therefore, the component $\eta$ imposes no additional restriction, and then
\[
\aut_\delta^\sigma(R)
=
\aut_{I_w}^\sigma(R)
\cap
\aut_\xi^\sigma(R).
\]

\end{proof}

\section{Jordanian Plane}\label{3}
The Jordanian plane is the associative $\Bbbk$-algebra
\[
\Lambda_2(\Bbbk)
=
\Bbbk\langle x,y:yx=xy+y^2\rangle.
\]
Together with the quantum plane, it constitutes one of the fundamental classes of two-generated quadratic algebras. Although these algebras share several structural features, Shirikov proved, among other results, that they are not
isomorphic; see \cite[Theorem 1.4]{Shirikov}. The Jordanian plane has been studied from several points of view, including automorphisms, derivations, modules, prime ideals, noncommutative algebraic geometry, Hopf algebras, and
Nichols algebras; see, for instance, \cite{Samuel}.

In this section, we study isotropy groups of derivations of
$\Lambda_2(\Bbbk)$. After recalling its PBW basis and the descriptions of its automorphisms and derivations, we obtain a criterion for the isotropy of an arbitrary inner derivation and apply it to monomial inner derivations. We also determine the isotropy groups of locally nilpotent derivations, which are precisely the derivations defined by $\delta_p(y)=0$, $\delta_p(x)=p(y)$, and, finally, use these computations to compare the resulting isotropy groups with those arising in the quantum Weyl algebra and in the first Weyl algebra.

\begin{lemma}{(Proposition 1.2,\cite{Shirikov})} \label{basis}
The basis of $\Lambda_2(\Bbbk)$ is $$\{x^iy^j\ | i,j \in \mathbb{N}_0 \}.$$ 

In particular, 
\[
y^mx^n= \sum_{l=0}^{n} {n\choose l}\frac{(m+n-l-1)!}{(m-1)!} x^ly^{m+n-l}; \  m,n \in \mathbb{N},
\]
and $\Lambda_2(\Bbbk)$ is a domain.
\end{lemma}

\begin{lemma}\label{lemad}{(Theorem 4.6,\cite{Shirikov})} If $char(\Bbbk)=0$, then each derivation $d$ of $\Lambda_2(\Bbbk)$ can be represented in
the form
\[
\delta(x) = \alpha y + \psi(x) + \add_w(x),\ \delta(y)= \psi'(x)y + \add_w(y)
\]    
for some $\alpha \in \Bbbk$, $\psi \in \Bbbk[x]$, $w \in  \Lambda_2(\Bbbk)$ and $\add_w(a) = wa - aw$, for $a \in \Lambda_2(\Bbbk)$. 
\end{lemma}

\begin{lemma}{(Theorem 3.1,\cite{Shirikov})}
If $char(\Bbbk)=0$ and $\varphi$ is an automorphism of the algebra
$\Lambda_2(\Bbbk)$, then 
\[
\varphi(x) = \gamma x + g(y), \  \varphi(y) = \gamma y,
\]
for some $\gamma \in \Bbbk^{*}$ and $g(y) \in \Bbbk[y]$.
\end{lemma}

An immediate consequence of the previous theorem is that 
\[
\aut(\Lambda_2(\Bbbk)) \ {\cong} \ \Bbbk^*\times \Bbbk[y]
\]
with respect to the operation $\circ$ such that: 
\[
(\gamma_2,g_2(y)) \circ (\gamma_1,g_1(y)) = (\gamma_1 \gamma_2, \gamma_1 g_2(y)+ g_1(\gamma_2y)).
\]

Accordingly, we denote by $\varphi=\varphi_{\gamma,g}$ whenever necessary.

\begin{lemma}
    Let $f\in\Lambda_2(\Bbbk)$ be any element. Then, $yf=fy$ if and only if $f\in \Bbbk[y]\cong \Bbbk[y][1;\delta]$. Also, $xf=fx$ if and only if  $f\in \Bbbk[x]\cong \Bbbk[1][x;\delta]$. In particular, $yf=fy$ and $xf=fx$ if and only if $f\in \Bbbk$.
\end{lemma}
\begin{proof}
    Note that it is trivial that if $f\in \Bbbk[y]$, then $yf=fy$. So, let $f=\sum_{ij}c_{ij}x^iy^j$. Then
    \begin{align*}
        yf
            &=\sum_{ij}c_{ij}yx^iy^j\\
            &= \sum_{ij}\sum_{l=0}^ic_{ij}\dfrac{i!}{l!}x^iy^{i+1-l+j}\\
            &= \sum_{ij}\sum_{l=0}^{i-1}c_{ij}\dfrac{i!}{l!}x^iy^{i+1-l+j}+fy.\\
    \end{align*}

    Let's denote $k=j-l$. Thus $j=k+l$ and $l=j-k$. As $0\le l\le i-1$, we obtain $j-i+1\le k\le j$. Thus, the previous expression becomes
    $$yf-fy=\sum_{ij}\sum_{k=j-i+1}^jc_{ij}\dfrac{i!}{(j-k)!}x^iy^{i+k+1}.$$
    
Note that the monomial $x^iy^{i+k+1}$ only appears once in the above expression and has coefficient $c_{ij}\dfrac{i!}{(j-k)!}$. Furthermore, note that the second summation has terms only for $i>0$. In this case, since $yf-fy=0$, then $c_{ij}=0$ for all $i>0$, which implies that $f\in k[y]$.

For the second statement, note that if $f\in \Bbbk[x]$, then $xf=fx$. So again let $f=\sum_{ij}c_{ij}x^iy^j$. Then
    \begin{align*}
        fx&=\sum_{ij}c_{ij}x^iy^jx\\
          &=\sum_{ij}c_{ij}x^i(jy^{j+1}+xy^j)\\
          &=\sum_{ij}c_{ij}x^ijy^{j+1}+xf.
    \end{align*}
That is, if $fx=xf$, then $c_{ij}=0$ for all $j>0$, that is, $f\in \Bbbk[x]$. Finally, the last statement follows from the previous two.
\end{proof}

\begin{proposition} \label{inner} Let $\delta= \add_w \in \der(\Lambda_2(\Bbbk))$ be an inner derivation with $w \in \Lambda_2(\Bbbk)$ and $\rho \in \au(\Lambda_2(\Bbbk))$. Then, $\rho \in \au_\delta(\Lambda_2(\Bbbk))$ if and only if $\rho(w)-w \in \Bbbk.$
\end{proposition}
\begin{proof}
    Using the notation of the previous theorem: $$\rho(x) = \gamma x + g(y), \  \rho(y) = \gamma y,$$ with $\gamma \in \Bbbk^{*}$ and $g(y) \in \Bbbk[y]$.      Additionally, if $\rho \in \au_\delta(\Lambda_2(\Bbbk))$, the isotropy group, then 
\begin{equation}\label{eq1}
\rho(\delta(x))=\delta(\rho(x)),
\end{equation}

\begin{equation}\label{eq2}
\rho(\delta(y))=\delta(\rho(y)).
\end{equation}

From the equation \eqref{eq2}, we have $(\rho(w)-w)y=y(\rho(w)-w)$ and then, by the previous lemma, $\rho(w)-w \in \Bbbk[y].$ From the equation \eqref{eq1}, we have $\gamma(\rho(w)-w)x+(\rho(w)-w)g(y)=\gamma x(\rho(w)-w)+g(y)(\rho(w)-w)$ and thus, by the previous lemma, $\rho(w)-w \in \Bbbk[x].$ Therefore, $\rho(w)-w \in \Bbbk$.
\end{proof}

\begin{remark} Let $h(x) \in \Bbbk[x] \setminus \Bbbk$. The algebra $A_h$ is the unital associative algebra over $\Bbbk$ with generators $x,y$ defining relation $yx-xy=h(x)$. Note that, up to isomorphism, the Jordanian plane is a class of these algebras. Kaygorodov, Lopes, and Mashurov (Proposition 2., \cite{ISF2021}) showed that $\delta \in \der(A_h)$ is locally nilpotent if and only if there exists $p(y) \in \Bbbk[y]$ such that $\delta(x)=p(y)$ and $\delta(y)=0$. Notice that a derivation of this form can be described as the Lemma \ref{lemad}. Indeed, denote $p(y)=\sum_{j=0}^rp_jy^j$ and let $w=\sum_{j=1} \frac{p_{j+1}}{j}y^j$, $\psi(x)=p_0$ and $\alpha=p_1$, thus:
\[
\delta^*(x) = \alpha y + \psi(x) + \add \ w(x),\ \delta^*(y)= \psi'(x)y + \add \ w(y)
\]    

Note that $\delta^*(y)=0$. We obtain $\delta^*(x)=p(y)$: in fact, $$\add_w(x)=\sum_{j=1}\frac{p_{j+1}}{j}(y^jx-xy^j)=$$$$=\sum_{j=1}\frac{p_{j+1}}{j}(jy^{j+1}+xy^j-xy^j)=\sum_{j=1} p_{j+1}y^{j+1}.$$

Therefore, $\delta^*(x)=p_1y+p_0+\sum_{j=1} p_{j+1}y^{j+1}=p(y).$
\end{remark}

\begin{proposition}\label{lnd}
Let $p(y)\in\Bbbk[y]$ and let $\delta_p$ be the locally nilpotent derivation of $\Lambda_2(\Bbbk)$ defined by $\delta_p(y)=0$ and
$\delta_p(x)=p(y)$. Then,
\[
\aut_{\delta_p}(\Lambda_2(\Bbbk))
=
\{\varphi_{\gamma,g}\in\aut(\Lambda_2(\Bbbk)) \, \vert \,
p(\gamma y)=\gamma p(y)\}.
\]
\end{proposition}

\begin{proof}
Let $\varphi_{\gamma,g}(x)=\gamma x+g(y)$, $\varphi_{\gamma,g}(y)=\gamma y$ be an automorphism of $\Lambda_2(\Bbbk)$. Since $\delta_p(y)=0$, we have $\varphi_{\gamma,g}(\delta_p(y))=0=
\delta_p(\varphi_{\gamma,g}(y))$. On the other hand,
\[
\varphi_{\gamma,g}(\delta_p(x))
=
\varphi_{\gamma,g}(p(y))
=
p(\gamma y),
\]
and
\[
\delta_p(\varphi_{\gamma,g}(x))
=
\delta_p(\gamma x+g(y))
=
\gamma p(y),
\]
since $\delta_p(g(y))=0$. Then, $\varphi_{\gamma,g}\delta_p=\delta_p\varphi_{\gamma,g}$ if and only if $p(\gamma y)=\gamma p(y)$. 
\end{proof}

\begin{proposition}\label{monomial}
Let $w=x^my^n$, with $m,n\geq0$ and let $\delta=\add_w$. Then, the following hold:

\begin{enumerate}
\item If $m=n=0$, then $\delta=0$ and
\[
\aut_\delta(\Lambda_2(\Bbbk))=\aut(\Lambda_2(\Bbbk)).
\]

\item If $m=0$ and $n>0$, then
\[
\aut_\delta(\Lambda_2(\Bbbk))
=
\{\varphi_{\gamma,g} \, \vert \, \gamma^n=1,\ g(y)\in\Bbbk[y]\}.
\]

\item If $m=1$ and $n=0$, then
\[
\aut_\delta(\Lambda_2(\Bbbk))
=
\{\varphi_{1,c}\, \vert \,c\in\Bbbk\}.
\]

\item If $m\geq1$ and $(m,n)\neq(1,0)$, then
\[
\aut_\delta(\Lambda_2(\Bbbk))
=
\{\varphi_{\gamma,0}\, \vert \,\gamma^{m+n}=1\}.
\]
\end{enumerate}
\end{proposition}

\begin{proof}
Due to Proposition \ref{inner}, we have $\varphi_{\gamma,g}\in\aut_\delta(\Lambda_2(\Bbbk))$
if and only if $\varphi_{\gamma,g}(w)-w\in\Bbbk$. If $m=n=0$, then $w=1$ and $\delta=\add_1=0$. Therefore, $\aut_\delta(\Lambda_2(\Bbbk))=\aut(\Lambda_2(\Bbbk))$. Suppose that $m=0$ and $n>0$. Then $w=y^n$ and $\varphi_{\gamma,g}(w)-w
=(\gamma^n-1)y^n$. This element belongs to $\Bbbk$ if and only if $\gamma^n=1$. Since no condition is imposed on $g(y)$, we obtain
\[
\aut_\delta(\Lambda_2(\Bbbk))
=\{\varphi_{\gamma,g} \, \vert \,\gamma^n=1,\ g(y)\in\Bbbk[y]\}.
\]

Suppose now that $m=1$ and $n=0$. Since $w=x$, we have
\[
\varphi_{\gamma,g}(w)-w
=
(\gamma-1)x+g(y).
\]
By the PBW basis, this element belongs to $\Bbbk$ if and only if $\gamma=1$ and $g(y)=c\in\Bbbk$. Then,
\[
\aut_\delta(\Lambda_2(\Bbbk))
=\{\varphi_{1,c}\, \vert \,c\in\Bbbk\}.
\]

It remains to consider $m\geq1$ and $(m,n)\neq(1,0)$. First, let $m=1$ and $n>0$. Then, $\varphi_{\gamma,g}(xy^n)-xy^n
=(\gamma^{n+1}-1)xy^n+\gamma^n g(y)y^n$. By the PBW basis, this element is scalar if and only if $\gamma^{n+1}=1$
and $g(y)=0$.

Assume now that $m\geq2$. Consider the filtration of
$\Lambda_2(\Bbbk)$ by $x$-degree,
\[
F_d=\bigoplus_{i=0}^{d}x^i\Bbbk[y].
\]

The defining relation $yx=xy+y^2$ implies that, for every $r\geq1$, $g(y)x^r-x^rg(y)\in F_{r-1}$. Therefore,
\[
(\gamma x+g(y))^m
=
\gamma^m x^m
+
m\gamma^{m-1}x^{m-1}g(y)
+
u,
\]
for some $u\in F_{m-2}$. It follows that
\[
\begin{aligned}
\varphi_{\gamma,g}(x^my^n)-x^my^n
&=
(\gamma^{m+n}-1)x^my^n \\
&\quad
+
m\gamma^{m+n-1}x^{m-1}g(y)y^n
+
\gamma^nuy^n,
\end{aligned}
\]
where $uy^n\in F_{m-2}$. Suppose that $\varphi_{\gamma,g}(x^my^n)-x^my^n\in\Bbbk$. Since $m\geq2$, comparison in $F_m/F_{m-1}$ gives $\gamma^{m+n}=1$. Comparison in $F_{m-1}/F_{m-2}$ then gives
\[
m\gamma^{m+n-1}g(y)y^n=0,
\]
it follows that $g(y)=0$.

Conversely, if $g(y)=0$ and $\gamma^{m+n}=1$, then $\varphi_{\gamma,0}(x^my^n)
=\gamma^{m+n}x^my^n=x^my^n$. Thus, $\varphi_{\gamma,0}\in
\aut_{\delta}(\Lambda_2(\Bbbk))$. Therefore,
\[
\aut_\delta(\Lambda_2(\Bbbk))
=
\{\varphi_{\gamma,0} \, \vert \,\gamma^{m+n}=1\}.
\]
\end{proof}

\begin{corollary}\label{cor}
Let $m\geq1$ and $n\geq0$, with $(m,n)\neq(1,0)$. Then
\[
\aut_{\add_{x^my^n}}(\Lambda_2(\Bbbk))
\cong
\mathbb Z_{m+n}.
\]
\end{corollary}




If two $\Bbbk$-algebras are isomorphic then there is a natural correspondence between their isotropy groups. More precisely, let $\rho: A \to B$ be an isomorphism of $\Bbbk$-algebras and $\delta:A \to A$ be a derivation then $\au_\delta \cong  \au_{\delta^*}$, where $\delta^*=\rho \circ \delta \circ \rho^{-1}$. Consequently, the collection of isotropy groups arising from derivations is an invariant of the algebra. In particular, if one algebra admits a derivation whose isotropy group has a feature that cannot occur as the isotropy group of any derivation of another algebra, then the two algebras cannot be isomorphic. Although the non-isomorphism considered below is classical, this point of view suggests isotropy groups as a possible strategy for distinguishing algebras.

\begin{corollary}\label{cor} $\Lambda_2(\Bbbk) \not\cong A^{1}_q(\Bbbk)$.

\end{corollary}

\begin{proof} 
Consider the locally nilpotent derivation
\[
\delta_1(y)=0,
\quad
\delta_1(x)=1
\]
of the Jordanian plane $\Lambda_2(\Bbbk)$. By Proposition
\ref{lnd},
\[
\aut_{\delta_1}(\Lambda_2(\Bbbk))
=
\{\varphi_{1,g} \, \vert \, g(y)\in\Bbbk[y]\}
\cong(\Bbbk[y],+).
\]

In particular, this isotropy group contains an infinite-dimensional triangular subgroup. On the other hand, the isotropy group of any ordinary derivation of $A_q^1(\Bbbk)$ is contained in $\aut(A_q^1(\Bbbk))\cong\Bbbk^*$. Thus, every such isotropy group is contained in a one-dimensional torus. This would identify the infinite-dimensional triangular isotropy group $\{\varphi_{1,g}\, \vert \,g(y)\in\Bbbk[y]\}$ with an isotropy subgroup contained in the one-dimensional torus $\Bbbk^*$, which is impossible in the natural ind-group structure on automorphism groups. Therefore,
\[
\Lambda_2(\Bbbk)\not\cong A_q^1(\Bbbk).
\]
\end{proof}

\begin{remark} Let $A_1:= \langle x, y \ \vert \ yx-xy=1 \rangle $ be the first Weyl algebra. If $\Bbbk$ is a field of characteristic zero, J. Dixmier (\cite{Dixmier}) proved that the group $\aut(A_1)$ is generated by its subgroups ${\rm Aff}(A_1)$, affine automorphisms, and ${\rm U}(A_1):=\{\phi_f: x \to x, \ y \to y+f \,\,\vert \, f\in \Bbbk[x]\}$, triangular automorphisms. In case of characteristic $p>0$, Makar-Limanov (\cite{Limanov}) proved that $\aut(A_1)$ and $\Gamma:= \{ \tau \in \aut(\Bbbk[x,y]) \ \vert \ \mathcal{J} (\tau)=1 \}$ are isomorphic as abstract groups in which $\mathcal{J} (\tau)$ is the Jacobian of $\tau$ and also presented a new proof of the case of characteristic zero. Furthermore, the derivations of $A_1$ are all inner when $char(\Bbbk)=0$. 
\end{remark}

The behavior described above should be contrasted with the classical first Weyl algebra. While the quantum Weyl algebra $A_q^1(\Bbbk)$, for $q$ not a root of
unity, has only diagonal automorphisms, the first Weyl algebra $A_1$ admits a much larger automorphism group. This additional flexibility allows isotropy
groups of derivations to contain large triangular subgroups: the next example shows this already for a very simple inner derivation.

\begin{example}
    Let $\delta^*=\add_x \in Der(A_1)$, $\rho:=(ax+by, cx+dy)$ an affine automorphism with $ad-bc = 1$ and $\varphi:=(x, y+f(x)) \in {\rm U}(A_1)$ a triangular automorphism. We also suppose that $\rho, \varphi \in \au_{\delta^*}(A_1)$, so: 
\begin{equation}\label{eq:isotropy-affine}
\delta^*(\rho(x))=\rho(\delta^*(x))
\quad\text{and}\quad
\delta^*(\rho(y))=\rho(\delta^*(y)),
\end{equation}

\begin{equation}\label{eq:isotropy-triangular}
\delta^*(\varphi(x))=\varphi(\delta^*(x))
\quad\text{and}\quad
\delta^*(\varphi(y))=\varphi(\delta^*(y)).
\end{equation}

From equation \eqref{eq:isotropy-affine}, 
\[
(ax+by)(ax+by)-(ax+by)(ax+by)=a(xx-xx)+b(xy-yx),
\]
thus, $b=0$. In addition, 
\[
(ax+by)(cx+dy)-(cx+dy)(ax+by)=c(xx-xx)+d(xy-yx),
\]
thus, $ad(xy-yx)=-d$. Note that $d \neq 0$, since $b=0$, and then $a=1$. Therefore, $\au_{\delta^*}(A_1)$ contains an infinite subgroup of the form 
\[
\{(x, cx+dy) \, \vert \, \ d = 1, \ c \in \Bbbk\} \cong (\Bbbk, +).
\]

From equation \eqref{eq:isotropy-triangular}, 
\[
\delta^*(\varphi(x))=\delta^*(x)=0=\varphi(\delta^*(x)),
\]

In addition, 
\[
\varphi(\delta^*(y))=-1=\delta^*(\varphi(y))=
\]
\[
=\delta^*(y+f(x))=xy+xf(x)-yx-f(x)x,
\]

Therefore, $\au_{\delta^*}(A_1)$ contains an infinite non-algebraic subgroup of the form $\{(x, y+f(x)) \vert f(x) \in \Bbbk[x]\}$: more precisely, automorphisms that preserve the pencil of lines $x=constant$. 

\end{example}

\bibliographystyle{abbrv}

\bibliography{refsss}

\bigskip

\vspace{5mm}
\small

\ttfamily ADRIANO DE SANTANA, UNIVERSIDADE TECNOLÓGICA FEDERAL DO PARANÁ, UTFPR TOLEDO/PR, BRASIL.

\textit{E-mail address:} adrianosantana@utfpr.edu.br 
\vspace{3mm}

\ttfamily RENE BALTAZAR, UNIVERSIDADE FEDERAL DO RIO GRANDE, FURG, SANTO ANTÔNIO DA PATRULHA/RS, BRASIL.

\textit{E-mail address:} renebaltazar.furg@gmail.com

\vspace{3mm}
\ttfamily ROBSON VINCIGUERRA, UNIVERSIDADE TECNOLÓGICA FEDERAL DO PARANÁ, UTFPR TOLEDO/PR, BRASIL.

\textit{E-mail address:} robsonwv@gmail.com
\vspace{3mm}

\ttfamily WILIAN DE ARAUJO, UNIVERSIDADE TECNOLÓGICA FEDERAL DO PARANÁ, UTFPR TOLEDO/PR, BRASIL.

\textit{E-mail address:} wilianmat@yahoo.com.br

\end{document}